\documentclass[11pt,letterpaper]{amsart}
\usepackage{graphicx}
\usepackage{amsmath, amssymb, amsthm, tikz, mathrsfs,verbatim,tensor,pstricks}
\usepackage[left=1in,top=1in,right=1in]{geometry}
\usepackage{setspace}
\usepackage{enumerate}

\usepackage{amsfonts}
\usepackage{latexsym}
\usepackage{graphics}

\usepackage{bbm}

\usepackage{etoolbox}
\patchcmd{\section}{\scshape}{\bfseries}{}{}
\makeatletter
\renewcommand{\@secnumfont}{\bfseries}
\makeatother

\newtheorem{theorem}{Theorem}
\newtheorem{lemma}{Lemma}
\newtheorem{corollary}[theorem]{Corollary}

\newtheorem{proposition}[theorem]{Proposition}
\newtheorem{definition}[theorem]{Definition}

\theoremstyle{definition}

\theoremstyle{plain}

\newcommand{\beqlbl}{\begin{equation}}
\newcommand{\eeqlbl}{\end{equation}}

\newcommand{\N}{\ensuremath{\mathbb{N}} }

\newcommand{\Z}{\ensuremath{\mathbb{Z}} }

\renewcommand{\P}{\mathbb{P}}

\newcommand{\bs}{\mathbf{s}}
\newcommand{\br}{\mathbf{r}}

\newcommand{\cU}{\mathcal{U}}

\usepackage{hyperref}
\usepackage{tikz}
\usepackage{xypic} 
\xyoption{curve}

\usetikzlibrary{decorations.markings}
\tikzstyle{vertex}=[circle, draw, inner sep=0pt, minimum size=6pt]
\tikzstyle{Vertex}=[circle, draw, inner sep=0pt, minimum size=14pt]
\tikzstyle{Vertexc}=[circle, draw, inner sep=0pt, minimum size=14pt, fill=blue!30]
\tikzstyle{vertexc}=[circle, draw, inner sep=0pt, minimum size=6pt, fill=red!40]
\tikzstyle{vertexcg}=[circle, draw, inner sep=0pt, minimum size=6pt, fill=green!70!black]

\newcommand{\beq}{\begin{equation*}}
\newcommand{\eeq}{\end{equation*}}
\newcommand{\ba}{\begin{align*}}
\newcommand{\ea}{\end{align*}}

\newcommand{\matbegin}[1]{\left (  \begin{array} {#1} }
\newcommand{\matend}{ \end{array} \right ) }

\newcommand{\bt}{\mathbf{t}}

\newcommand{\Ht}{\mathrm{ht}}
\newcommand{\fT}{\mathfrak{T}}

\newcommand{\avn}[1]{\textbf{A\!v}_n(#1)}

\begin{document}

\title[Local limits of permutations avoiding a pattern of length three]{A Galton-Watson tree approach to local limits of permutations avoiding a pattern of length three}

\author[Jungeun~Park]{ \ Jungeun~Park}
\author[Douglas~Rizzolo]{ \ Douglas~Rizzolo}

\thanks{
Supported in part by NSF grant DMS-1855568
}

\vskip1.3cm

\begin{abstract}
We use local limits of Galton-Watson trees to establish local limit theorems for permutations conditioned to avoid a pattern of length three.  In the case of ${\bf 321}$-avoiding permutations our results resolve an open problem of Pinsky.  In the other cases our results give new descriptions of the limiting objects in terms of size-biased Galton-Watson trees.           
\end{abstract}

\maketitle
\section{Introduction}
If $\pi$ and $\sigma$ are permutations of $[n]=\{1,\dots, n\}$ and $[m]$ respectively with $n>m$, then $\pi$ is said to contain $\sigma$ if there exist indices $i_1<\dots< i_m$ such that for all $j<k$, $\pi(i_j)<\pi(i_k)$ if and only if $\sigma(j)<\sigma (k)$.  The permutation $\pi$ is said to avoid $\sigma$ if it does not contain $\sigma$.  Recently there has been considerable interest in the study of random pattern-avoiding permutations.  Initially there was a focus on the case when $m=3$, $\sigma$ was a fixed permutation of $\{1,2,3\}$, and $n$ tended to infinity, which was motivated by studying the longest increasing subsequence problem for pattern avoiding permutations \cite{MR2343720,notknuth,Sniady}.  Since the initial work, the focus has expanded to the understanding the rich behavior of permutations avoiding one pattern, or several patterns, see e.g.\! \cite{borga2018localsubclose, zbMATH07527828, zbMATH07481271, borga2019square, hoffman2017pattern,  HRS1, hoffman2019scaling, Ja_multiple, ML, madras_yildirim,  mrs, zbMATH07749502} for a small sample of the recent literature illustrating the broad range phenomena pattern avoiding permutations exhibit.

Even with all of the activity in the field, permutations avoiding a pattern of length three continue to be a source of interesting behavior.  In the present paper, we are concerned with local limits of permutations avoiding a pattern of length three.  That is, for a fixed pattern $\sigma$ of length three, we let $\Pi_n$ be a uniformly random permutation of $[n]$ conditioned to avoid $\sigma$ and for each $k$ we consider the convergence in distribution of $(\Pi_n(1),\dots, \Pi_n(k))$ as $n$ tends to infinity.  Local limits of this type were used implicitly in \cite{HRS2}, but they were first directly studied in \cite{pinsky2020infinite}.  One of the motivations in \cite{pinsky2020infinite} for studying this type of local limit is to investigate how the limits fit within the recently developed theory of infinite regenerative permutations \cite{zbMATH07067272}.  We remark that local limits of a different nature were considered in \cite{Borga2019}.  In \cite{pinsky2020infinite}, local limits are obtained for $\sigma \neq {\bf 321}$, but only partial results were obtained for $\sigma= {\bf 321}$ and the question of whether or not the local limit existed was left open.  In the present paper we answer this question and show that the local limit does, in fact, exist.  Our methods give a unified approach that works when $\sigma$ is any pattern of length three.

Our main result is most naturally stated by considering permutations of $[n]$ as functions from $\N$ to its one-point compactification $\N^*= \N\cup\{\infty\}$.  To formalize this, let $S_n$ be the set of permutations of $[n]$, let $\avn{\sigma}$ be the subset of $S_n$ of permutations avoiding $\sigma$, and let $F(\N,\N^*)$ be the set of functions from $\N$ to $\N^*$ with the topology of pointwise convergence.  We embed $S_n$ in $F(\N,\N^*)$ by setting $\pi(k)=k$ for $k>n$. 

\begin{theorem}
Fix $\sigma\in S_3$ and suppose that $\Pi^\sigma_n$ is a uniformly random element of $\avn{\sigma}$.  Then there exists an $F(\N,\N^*)$-valued random variable $\Pi^\sigma$ such that $\Pi^\sigma_n \overset{d}{\longrightarrow} \Pi^\sigma$ as $n\to \infty$.
\end{theorem}

We are able to give very precise descriptions of each $\Pi^\sigma$, but we delay this to Theorem \ref{thm main precise} since we must first develop the appropriate notation.  Our methods can be viewed as an extension of the methods developed in \cite{HRS2}.  We rely on bijections with rooted ordered trees and we obtain local limits and descriptions of the limiting objects using deterministic arguments based on the local convergence of Galton-Watson trees conditioned on their number of vertices to size-biased Galton-Watson trees.  In contrast, the methods used in \cite{pinsky2020infinite} are based on direct enumeration and regenerative properties of pattern-avoiding permutations.
 
\section{Tree formalism} 
In this section we introduce the a formalism for trees and tree convergence that will let us give a unified approach to infinite limits of permutations avoiding a pattern of length three.  The formalism for trees we adopt is Neveu's formalism for rooted ordered trees \cite{N86}.  This will give us a consistent way to refer to the vertices of different trees.  We let
\[ \cU = \bigcup_{n\geq 0} \{1,2,\dots\}^n,\]
where $\{1,2,\dots\}^0 = \{\emptyset\}$.  For $u\in\cU$, we let $|u|$ be the length of $u$.  That is, $|u|$ is the unique integer such that $u\in \{1,2,\dots\}^{|u|}$.  We treat the elements of $\cU$ as lists, and for $u,v\in \cU$, we denote the concatenation of $u$ and $v$ by $uv$.  The empty list is the identity for this operation, that is, if $u=\emptyset$ then $uv=vu=v$.  The set of ancestors of $u$ is defined to be
\[ A_u = \{ v \in \cU  : \exists w\in \cU \textrm{ such that } u=vw\}\]

We can consider $\cU$ as the vertex set of a rooted ordered tree $U_\infty$, called the Ulam-Harris tree, where we take the lexicographic order on $\cU$ as the order, $\emptyset$ as the root, and for every $u\in \cU$ and $i\in \{1,2,3,\dots\}$ there is an edge from $u$ to $ui$.  By a rooted, ordered tree we will mean a (possibly infinite) subtree $\bt$ of $U_\infty$ such that
\begin{enumerate}
\item $\emptyset\in \bt$.
\item If $u\in\bt$ then $A_u\subseteq \bt$.
\item For every $u\in  \bt$, if $ui\in \bt$ then $uj\in \bt$ for all $j \leq i$.
\end{enumerate}

We denote the number of vertices in $\bt$ by $|\bt|$.  A tree $\bt\subseteq U_\infty$ is considered as rooted and ordered by taking $\emptyset$ to be the root and using the left-to-right order on $\bt$ induced by the lexicographic order on $\cU$.  The lexicographic order of $\cU$ will be denoted by $\leq$.  If $u\in \bt$, the vertices in $A_u$ are the vertices in $\bt$ on the path from the root to $u$.  For $u\in\bt$ the height of $u$, denoted $\Ht(u)$, is the number of edges on the path from the root of $\bt$ to $u$.  A consequence of our formalism is that $\Ht(u)= |u|$.  For $u\in \bt$ the set of descendants of $u$, also called the fringe subtree at $u$, is defined as
\[ \bt_u= \{v\in \bt: \exists w \in \cU \textrm{ such that } v=uw\}.\]
If $v$ is a descendant of $u$ and $|v|=|u|+1$ we will say that $v$ is a child of $u$ and $u$ is the parent of $v$ -- note that in this case there exists some $i\in \{1,2,\dots\}$ such that $v=ui$.  

For a tree $\bt$ and $u\in \bt$, define $d_\bt(u)=0 \vee \sup\{ i : ui\in\bt\}$.  We also define $d_\bt(u)=-1$ if $u\notin \bt$.  A vertex $u\in \bt$ is called a leaf if $d_\bt(u)=0$.  This differs slightly from the typical definition of a leaf because the root is not considered a leaf unless the tree has only on vertex, even though the root may have only one edge incident to it.   The number of leaves in $\bt$ will be denoted by $L(\bt)$.  

A tree $\bt$ is called \textit{locally finite} if $d_\bt(u)<\infty$ for all $u\in\cU$ and define
\[ \fT = \{ \bt\subseteq U_\infty : \bt \textrm{ is locally finite}\}.\]

If $\bt\in \fT$ and $v\in  \cU$, we define the re-rooted tree $v\bt$ by
\[ v\bt =  \{ vu \in \cU : u\in \bt\}.\]
Note that $v\bt \notin \fT$ if $v\neq \emptyset$ since, for example, $\emptyset \notin v\bt$.  $v\bt$ is still a subtree of the Ulam-Harris tree $U_\infty$ in the regular graph-theoretic sense, but we take its root to be $v$ instead of $\emptyset$.  For $\bs,\bt\in \fT$ and $v\in \bt$, we say that $\bs$ is attached to $\bt$ at $v$ if for some $i\geq 1$ 
\[ vi\bs = \bt_{vi}= \{u\in \bt : u=viw \textrm{ for some } w\in \cU\}.\]
If $i$ is important we will say that $\bs$ is the $i$'th subtree from the left attached to $v$.  This is something of an abuse of terminology because $\bs$ need not be a subtree of $\bt$ in the graph-theoretic sense, but rather it is isomorphic to a subtree of $\bt$ (specifically, the subtree $vi\bs$).

For a tree $\bt\in \fT$ and a leaf $l\in \bt$ we define the attachment operator $\langle \cdot \rangle_{\bt,l} : \cup_{k\geq 1} \fT^k \to \fT$ by defining $\langle \bt_1,\dots, \bt_k \rangle_{\bt,l}$ to be the subtree of $U_\infty$ with vertex set 
\[ \bt \cup \bigcup_{i=1}^k li\bt_i.\]
Note $l$ is a vertex in $\langle \bt_1,\dots, \bt_k \rangle_{\bt,l}$ with degree $k$ and $\bt_i$ is the $i$'th subtree from the left in $\langle \bt_1,\dots, \bt_k \rangle_{\bt,l}$.

\section{Local Limits of Trees}
We topologize $\fT$  with the convergence $\bt_n\rightarrow \bt$ if $d_{\bt_n}(u) \rightarrow d_{\bt}(u)$ for all $u\in \cU$.  Since $\cU$ is countable this is convergence can be metrized by, for example, the complete metric
\[ D(\bt,\mathbf{s}) = \sum_{i=1}^\infty \frac{ |d_{\bt}(f(i)) - d_{\mathbf{s}}(f(i))| \wedge 1}{2^i},\]
where $f: \{1,2,\dots, \}\to \cU$ is a bijection.  For our purposes, it is useful to observe that $\bt_n\to \bt$ if and only if for each $m\geq 1$, $\bt_{n}^{[m]} = \bt^{[m]}$ for all sufficiently large $n$.

Let $\xi$ be a probability distribution on $\{0,1,2,\dots\}$ with mean $1$, ie. $\sum_{i} i\xi(i) =1$, and $\xi(1)<1$.  A $\fT$-valued random variable is called a Galton-Watson tree with offspring distribution $\xi$ if
\[ \P(T = \bt) = \prod_{u\in \bt} \xi(d_\bt(u))\]
for all $\bt \in \fT$.  For $\bt\in\fT$, let $\bt^{[k]}$ be the subtree of $\bt$ comprised of vertices of height at most $k$.  Let $T$ be a $\xi$-Galton-Watson tree and define the probability measure $\nu_\xi$ on $\fT$ to be the unique measure on infinite trees such that for every $k$ and $\bt_0\in \fT$ with height $k$, 
\[ \nu_\xi(\{ \bt \in \fT : \bt^{[k]} = \bt_0\}) = (\#_k\bt_0) \P( T^{[k]} = \bt_0),\]
where $\#_k\bt$ is the number of vertices in $\bt$ with height $k$.  See e.g. \cite[Lemma 1.14]{Kesten86} for the existence and uniqueness of this measure.    Let $\tilde T$ have distribution $\nu_\xi$.  $\tilde T$ is called a size-biased $\xi$-Galton-Watson tree.   The size biased distribution of $\xi$ is the distribution
\[ \tilde \xi(n) = n\xi(n).\]
Note that this is a probability distribution because $\xi$ has mean $1$.  Observe that, directly from the definition of $\nu_\xi$, we have
\[ \P(\#_1 \tilde T = k) = \tilde \xi(k),\]
which justifies calling $\tilde T$ the size-biased tree. 

The properties of $\tilde T$ that are most important for our purposes can most easily be seen from the constructions in Chapter 12 of \cite{LyonsPeres} and Section 5 of \cite{J12}.  In particular, we have the following sequential construction of $\tilde T$ from \cite[Chapter 12]{LyonsPeres}, which we have adapted to our tree formalism.  Let $\tilde T_0 = \emptyset$ and $\eta_0=\emptyset$.  Conditionally given $\tilde T_n$ and $\eta_0,\dots, \eta_n$ where $\eta_i$ is a leaf of $\tilde T_i$ at height $i$ and $\eta_{i+1}$ is a child of $\eta_i$, construct $\tilde T_{n+1}$ as follows.  Choose $K_n$ according to $\tilde \xi$ and, given $K_n$ let $i^*$ be chosen uniformly at random from $\{1,2,\dots, K_n\}$ and let $T_1, \dots T_{K_n}$ be independent $\xi$-Galton-Watson trees.  Define $\eta_{n+1}=\eta_n i^*$ and 
\[ \tilde T_{n+1} = \langle T_1,\dots, T_{i^*-1},\emptyset, T_{i^*+1},\dots T_{K_n}\rangle_{\tilde T_n,\eta_n}.\]
Note that $\eta_{n+1}$ is the $i^*$'th child from the left of $\eta_n$ in $\tilde T_{n+1}$, and is a leaf, while independent $\xi$-Galton-Watson trees are attached to the other $K_n-1$ children of $\eta_n$ in $\tilde T_{n+1}$.  Moreover, going from $\tilde T_n$ to $\tilde T_{n+1}$ involves adding descendants to $\eta_i$, but other parts of $\tilde T_n$ are not modified.  Furthermore, $\tilde T_0 \subset \tilde T_1\subset \tilde T_2\subset \cdots $.  Then, following  \cite[Chapter 12]{LyonsPeres}, we have that
\[ \tilde T =_d \cup_{i=0}^\infty \tilde T_i.\]

For our purposes, $\tilde T$ has two important properties, which we summarize in the next two lemmas.  The first is an immediate consequence of the discussion above.

\begin{lemma}\label{lemma sizebp}
With probability $1$, $\tilde T$ is locally finite and has a unique infinite path (called a spine) with infinitely many independent $\xi$-Galton-Watson trees attached to the left and right of the spine. 
\end{lemma}  

Indeed, the result of this Lemma can also be taken as the basis for a construction of $\tilde T$, see \cite[p. 115]{J12}.

\begin{lemma}\label{lemma local limit}
Let $T$ be a $\xi$-Galton-Watson tree and for $n$ such that $\P(\# T=n)>0$, let $T_n$ be distributed like $T$ conditioned on $\P(\# T=n)$.  Then
\[ T_n \overset{d}{\longrightarrow} \tilde T.\] 
That is, for every tree $\bt\in \fT$ and $k\in \{1,2,3,\dots\}$ we have
\[ \lim_{n\to\infty} \P\left(T_n^{[k]} = \bt\right)=\lim_{n\to\infty} \P\left(T^{[k]} = \bt \middle| \# T = n\right) = \P\left(\tilde T^{[k]} = \bt\right),\]
where the limit is understood to be along $n$ such that $\P(\# T=n)>0$.
\end{lemma}

This lemma has a long history.  The case when $\xi$ has finite variance, which is the only case we will need, was proven implicitly in \cite{K75} and first stated explicitly in \cite{AP98}.  A version for conditioning on $\#T\geq n$ rather than $\#T = n$ was given in \cite{Kesten86}.  See \cite[Theorem 7.1]{J12} and \cite{AD14} for modern approaches to the general result.

\section{Bijections between rooted plane trees and pattern avoiding permutations}\label{treebijection}
Let $\fT_f\subseteq \fT$ be the set of finite trees with at least two vertices and let $\fT_{n+1}\subseteq\fT_f$ denote the set of rooted plane trees with $n+1$ vertices.  In this section we introduce bijections from $\fT_{n+1}$ to $\avn{\sigma}$ for $\sigma \in S_3$.

\subsection{321/123-avoiding Permutations}
We will describe a bijection between $\fT_{n+1}$ and $\avn{\bf 321}$ that maps the leaves of the tree to the left-to-right maxima of the corresponding permutation.  This bijection was previously used in \cite{HRS2}.

Suppose $\bt$ has $k$ leaves and let $\ell(\bt) = (l_1,l_2,\cdots,l_k)$ be the list of leaves in $\bt$, listed in lexicographic order.  There are two quantities associated with the leaves of the tree needed to construct the bijection.  Let $s_i=s_{\bt,i} = |\{v\in\bt : v<l_i\}|$ equal the number of vertices in $\bt$ less than $l_i$ in the lexicographic order.  Let $p_i=p_{\bt,i} = |l_i|$ denote the height of $l_i$.  

Construct two increasing sequences of length $k$, $A = \{s_i\}_{i=1}^k$ and $B = \{ s_i - p_i + 1\}_{i=1}^k.$  We define $\Phi^{\bf 321}_\bt$ pointwise on $B$
$$\Phi^{\bf 321}_\bt( s_i-p_i + 1 ) = s_i$$ for $1\leq i \leq k$.  We then extend $\Phi^{\bf 321}_\bt$ to $[n]\backslash B$ by assigning values of $[n]\backslash A$ in increasing order.  

The resulting function $\Phi^{\bf 321}_\bt$ is a permutation in $\avn{321}$.  The set of values for $s_i$ and $p_i$ uniquely determine a rooted plane tree, so we may recover the corresponding rooted tree from the left-to-right maxima of a permutation in $\Phi^{\bf 321}_\bt$.  Let $M=\{m_i\}_{i=1}^k$ denote the set of indices of the left-to-right maxima of $\Phi^{\bf 321}_\bt$. Then
$$s_i = \Phi^{\bf 321}_\bt(m_i)$$
and 
$$p_i = \Phi^{\bf 321}_\bt(m_i) - m_i + 1.$$

Note that if $\pi \mapsto n+1-\pi$ is a bijection from $\avn{{\bf 321}}$ to $\avn{{\bf 123}}$.  This motivates defining $\Phi^{\bf 123}_\bt = |\bt|+1-  \Phi^{\bf 321}_\bt$.  This discussion yields the following result.

\begin{theorem}\label{thm bij1} 
 For $\sigma\in \{ {\bf 123}, {\bf 321}\}$, $\bt\mapsto\Phi^{\sigma}_\bt$ is a bijection from $\fT_{n+1}$ to $\avn{{\sigma}}$.
\end{theorem}

\subsection{231/213/312/132-avoiding Permutations}
In this section we introduce four related bijections from $\fT_{n+1}$ to $\avn{\bf 231}$, $\avn{\bf 213}$. $\avn{\bf 312}$, and $\avn{\bf 132}$. In the ${\bf 231}$-avoiding case this bijection first appeared in \cite{hoffman2017pattern}.  The remaining cases come from the natural symmetries relating permutations avoiding these patterns (e.g if $\pi \in \avn{\bf 231}$ then $i\mapsto n+1-\pi(i) \in \avn{\bf 213}$). For $\bt\in \fT_{n+1}$, let $v_0,v_1,\dots,v_n$ be the vertices of $\bt$ listed in lexicographic order (so that $v_0=\emptyset$ is the root) and let $w_i = v_{n+1-i}$.  We define $\Phi^{\bf 231}_{\bt}\in \avn{\bf 231}$ by  
\begin{equation}\label{eq 231bijection} \Phi^{\bf 231}_{\bt}(i) = i + |\bt_{v_i}| - |v_i|, \quad i=1,2,\dots, |\bt|,\end{equation}
where $|\bt|$ is the number of vertices of $\bt$. We define $\Phi^{\bf 213}_{\bt}(i) \in \avn{\bf 213}$ by
\begin{equation}\label{eq 213bijection} \Phi^{\bf 213}_{\bt}(i) = |\bt\setminus\bt_{v_i}| -i+ |v_i|, \quad i=1,2,\dots, |\bt|.\end{equation}
We define $\Phi^{\bf 312}_\bt(i) \in \avn{\bf 312}$ by
\begin{equation}\label{eq 312bijection} \Phi^{\bf 312}_\bt(i) = i-|\bt_{w_i}|+|w_i|, \quad i=1,2,\dots, |\bt|.\end{equation}
We define $\Phi^{\bf 132}_\bt \in \avn{\bf 132}$
\begin{equation}\label{eq 132bijection} \Phi^{\bf 132}_\bt(i) = |\bt|-i + |\bt_{w_i}| - |w_i|, \quad i=1,2,\dots, |\bt|.\end{equation}
\begin{theorem} \label{thm bij2} 
For $\sigma \in \{ {\bf 231}, {\bf 213}, {\bf 312}, {\bf 132}\}$ the map $\bt\mapsto\Phi^{\sigma}_{\bt}$ is a bijection from $\fT_{n+1}$ to $\avn{\sigma}$.
\end{theorem}

\subsection{Extension to Infinite Trees}
In this section we show that the bijections above can be extended in a natural way to be defined on certain infinite trees.  An infinite spine in $\bt$ is an infinite sequence $\eta_0,\eta_1,\dots$ of vertices in $\bt$ such that $\eta_0=\emptyset$ and $\eta_{i+1}$ is a child of $\eta_i$ for all $i\geq 0$.  Let $\fT_\infty\subseteq \fT$ be the set of locally finite infinite trees that have a unique infinite spine, have infinitely many trees attached to the left and right of the infinite spine, and such that infinitely many of these trees to the left of the infinite spine are in $\fT_f$.  Note that because $\bt\in \fT_\infty$ has a single infinite spine, the subtrees branching off from the infinite spine are finite.  Such a subtree is either in $\fT_f$ or it consists of a single vertex (in which case it becomes a leaf attached to the spine).  We will show that the formulas used to define the bijections in the previous section can be extended to $\fT_\infty$, though we must allow some of them to take infinite values and they may no longer be injective. 

It may not be possible to list the vertices of an infinite tree in increasing lexicographic order.  However, for any $\bt\in \fT$ we can still consider the infinite chain of vertices $(v_i)_{i<|\bt|}=(v_{\bt,i})_{i<|\bt|}$ defined by $v_{\bt,0}=\emptyset$ and 
\[v_i=v_{\bt, i} = \min\{v \in \bt : v\notin \{v_{\bt,0},v_{\bt,1},\dots,v_{\bt,i-1}\}\}\]
for $1\leq i<|\bt|$.  Note that if that if $\bt$ is finite then $(v_{\bt,i})_{i<|\bt|}$ is the list of vertices in $\bt$ in increasing lexicographic order.  

For $\bt\in \fT_\infty$, note that because there are infinitely many trees branching off to the right of the unique infinite spine, $\bt$ contains a largest vertex in the lexicographic order.  Indeed we can define $w_0=w_{\bt,0} = \max\{v \in \bt\}$ and for $i \geq 1$,
\[ w_i=w_{\bt,i} =   \max\{w\in \bt : w\notin \{w_{\bt,0},w_{\bt,1},\dots,w_{\bt,i-1}\}\}.\]
Note that none of the $w_{\bt,i}$ are on the infinite spine of $\bt$ because each vertex on the spine is smaller than the vertices in trees branching off to the right of the spine and there are infinitely many such vertices.

%Let $\N^* = \N\cup \{\infty\}$ and let $F(\N,\N^*)$ be the set of functions from $\N\to\N^*$.

\begin{definition}
For $\bt \in \fT_\infty$, we define functions $\Phi^{\bf 231}_{\bt}, \Phi^{\bf 213}_{\bt}, \Phi^{\bf 312}_\bt, \Phi^{\bf 132}_\bt \in F(\N,\N^*)$ by 
\begin{enumerate}
\item $\Phi^{\bf 231}_{\bt}(i) = i + |\bt_{v_i}| - |v_i|$, \vskip .2cm
\item $\Phi^{\bf 213}_{\bt}(i) = |\bt\setminus\bt_{v_i}| -i+ |v_i|$, \vskip .2cm
\item $\Phi^{\bf 312}_\bt(i) = i-|\bt_{w_i}|+|w_i|$, \vskip .2cm
\item $\Phi^{\bf 132}_\bt(i) = |\bt|-i + |\bt_{w_i}| - |w_i|$, 
\end{enumerate}
where we use the convention that $a+\infty=\infty$ for $a\in \Z \cup \{\infty\}$.
\end{definition}
Note that $-\infty$ never enters into the formulas in the above definition, so we do not need to worry about expressions like $\infty-\infty$.  Also, note that $\Phi^{\bf 132}_\bt(i)=\infty$ for all $i$ since $|\bt|=\infty$.  We chose the expression in the definition to make it clear that this definition is consistent with the case when $\bt \in \fT_f$.

The extension of $\Phi^{\bf 321}_{\bt}$ to $\fT_\infty$ is more delicate since it is not obvious that the procedure we used on finite trees is well defined.  The following proposition shows that, in fact, the procedure is well defined.  We first extent the notation to $\fT_\infty$.  For $\bt \in\fT_\infty$, let $(l_1, l_2,\dots)$ be a list of the leaves to the left of the infinite spine of $\bt$ in increasing lexicographic order.  As before we define $s_i=s_{\bt,i} = |\{v\in\bt : v<l_i\}|$ and $p_i=p_{\bt,i}=|l_i|$.

\begin{proposition}\label{prop bij}
If $\bt\in \fT_\infty$ there is a unique increasing bijection from $\{s_i-p_i+1: i\geq 1\}^c$ to $\{s_i : i\geq 1\}^c$.
\end{proposition} 

\begin{proof}
Let $(l_1, l_2,\dots)$ be a list of the leaves to the left of the infinite spine of $\bt$ in increasing lexicographic order.  Note that this list is infinite because we assume that there are infinitely many trees attached to the left of the spine of $\bt$. It is sufficient to show that $\{s_i : i\geq 1\}^c $ and $\{s_i-p_i+1: i\geq 1\}^c$ are infinite, from which it follows that there is a unique increasing bijection between these sets.  Since $\{s_i-p_i+1\}_{i\geq 1}$ is an increasing sequence, to show $|\{s_i-p_i+1: i\geq 1\}^c|=\infty$, it suffices to find infinitely many $i$ such that $s_{i+1}-p_{i+1}+1 > s_{i}-p_{i}+2$, since then $s_{i}-p_{i}+2 \notin \{s_i-p_i+1: i\geq 1\}$. 
Let 
\[I = \{ i: s_{i+1}-p_{i+1}+1 > s_{i}-p_{i}+2\}.\] 

Let $\eta_0,\eta_1,\eta_2,\dots$ be the vertices on the infinite spine of $\bt$, in increasing lexicographic order,  Suppose that $\bs$ is a tree in $\fT_f$, is the $i$'th subtree from the left attached to $\eta_k$, and is to the left of $\eta_{k+1}$.  Since $\bs\in \fT_f$, it contains at least two vertices.  Thus 
\[\max_{v\in \eta_ki\bs} d(v,\eta_k) \geq 2.\]  
Let $i^*$ be such that $l_{i^*}$ is the right-most leaf of $\bt$ in $\eta_ki\bs$.  Note that $d(l_{i^*},\eta_k) \geq 2$.  Since $l_{i^*}$ is the right-most leaf of $\bt$ in $\eta_ki\bs$, $\eta_k$ is on the path from $l_{i^*}$ to $l_{i^*+1}$.
%be such that the distance from $l_{i^*}$ to the infinite spine is at least $2$ and $d(l_{i^*},l_{i^*+1}) \geq 3$.  Note that this happens if either $l_{i^*+1}$ closer to the spine than $l_{i^*}$ or if the path from $l_{i^*}$ to $l_{i^*+1}$ contains an edge on the infinite spine.  Thus, since infinitely many of these trees to the left of the spine have height at least $2$ and each of these trees contains at least one leaf with these properties (in particular, the right-most leaf with height at least $2$), there are infinitely many indices $i$ with the properties required of $i^*$.  Let $v_{i^*}$ be the branch point where the paths from the root to $l_{i^*}$ and $l_{i^*+1}$ diverge.  Note that our assumptions on $i^*$ imply that $d(v_{i^*},l_{i^*})\geq 2$.   
Consequently
\[ \begin{split} s_{i^*+1} -p_{i^*+1} +1 & =  s_{i^*}+d(\eta_{k},l_{i^*+1}) - (p_{i^*}- d(\eta_{k},l_{i^*})+ d(\eta_k,l_{i^*+1}))+1\\
& = s_{i^*}-p_{i^*}+1 + d(\eta_{k},l_{i^*})\\
& \geq s_{i^*}-p_{i^*}+1 + 2 
.\end{split}\]
Therefore $i^*\in I$, so $|I|=\infty$ since infinitely many of the trees attached to the left of the spine of $\bt$ are in $\fT_f$, and it follows that $|\{s_i-p_i+1: i\geq 1\}^c|=\infty$.  Similarly, if $j^*$ is such that the path from $l_{j^*}$ to $l_{j^*+1}$ contains an edge on the infinite spine, then $s_{j^*+1} \geq s_{j^*}+2$, so $| \{s_i\}_{i\geq 1}|^c =\infty$.
\end{proof}

As a consequence of this Proposition, we can make the following definition.

\begin{definition}
For $\bt\in\fT$, define $\Phi^{\bf 321}_{\bt} \in F(\N,\N^*)$ by 
\[ \Phi^{\bf 321}_\bt( s_i-p_i + 1 ) = s_i \]
and extend $\Phi^{\bf 321}_{\bt}$ to $\N \setminus \{s_i-p_i + 1: i\geq 1\}$ using the unique increasing bijection from this set to  $\N \setminus \{s_i: i\geq 1\}$. Additionally, define $\Phi_{\bt}^{\bf 123} = |\bt|+1- \Phi_\bt^{\bf 321}$.
\end{definition}

Note that the increasing bijection used in this definition is exactly that whose existence in guaranteed by Proposition \ref{prop bij}.  Also note that $\Phi^{\bf 321}_t$ does not take infinite values and, in fact, is a bijection from $\N$ to $\N$.  Note also that $\Phi^{\bf 123}(k)=\infty$ for all $k$ and, again, we have given the definition that we have in order to emphasize that this definition is consistent with the definition for finite trees.

\section{Continuity of the bijections}
The central observations to proving continuity are that, as noted above, $\bt_n\to \bt$ if and only if for each $m\geq 1$, $\bt_{n}^{[m]} = \bt^{[m]}$ for all sufficiently large $n$ and that for each $k$ the tree statistics the go into determining the image of $k$ are either completely or approximately determined by  $\bt^{[m]}$.  The following results serve as important tools for translating this observation into careful arguments.

\begin{lemma}\label{lem reduction}
Suppose that $\bt \in \fT_\infty$ and $\bs\in \fT_f$.  If 
\[ m \geq 1+\max_{0\leq i\leq j} |v_{\bt,i}|.\]
and $\bs^{[m]} = \bt^{[m]}$ then 
\[ (v_{\bt,0}, v_{\bt,1},\dots,v_{\bt,j}) = (v_{\bs,0}, v_{\bs,1},\dots, v_{\bs,j}).\]
Similarly, if 
\[ m \geq 1+\max_{0\leq i\leq j} |w_{\bt,i}|.\]
and $\bs^{[m]} = \bt^{[m]}$ then 
\[ (w_{\bt,0}, w_{\bt,1},\dots,w_{\bt,j}) = (w_{\bs,0}, w_{\bs,1},\dots, w_{\bs,j}),\]
where recall that, for a finite tree, $w_{\bs,i}=v_{\bs,|\bs|+1-i}$.
\end{lemma}

This lemma says that the initial string of the lexicographically ordered vertices is determined by the portion of the tree below one more than the maximum height of these vertices.

\begin{proof}
Since $m \geq 1+\max_{0\leq i\leq j} |v_{\bt,i}|$ we see that $v_{\bt,0}, v_{\bt,1},\dots,v_{\bt,j} \in \bt^{[m]}$.  Since $\bs^{[m]} = \bt^{[m]}$, we see that $v_{\bt,0}, v_{\bt,1},\dots,v_{\bt,j} \in \bs^{[m]}$.  Note that $v_{\bt,0}=\emptyset=v_{\bs,0}$. We proceed to argue recursively.  Suppose $k<j$ is such that 
\[(v_{\bt,0}, v_{\bt,1},\dots,v_{\bt,k}) = (v_{\bs,0}, v_{\bs,1},\dots, v_{\bs,k}).\]
Since $v_{\bt,k+1}\in  \bs^{[m]}  \setminus \{v_{\bs,0}, v_{\bs,1},\dots,v_{\bs,k} \}$, we see that $ v_{\bs,k+1} \leq v_{\bt,k+1}$ (the $\leq$ being the lexicographic order).  From the definition of a tree, we see that $ v_{\bs,k+1}$ is adjacent to one of $v_{\bs,0}, v_{\bs,1},\dots, v_{\bs,k}$ and, consequently,
\[ |v_{\bs,k+1}| \leq 1+\max_{0\leq i\leq k} |v_{\bs,i}| =1+\max_{0\leq i\leq k} |v_{\bt,i}| \leq 1+\max_{0\leq i\leq j} |v_{\bt,i}| \leq m.\]
Thus $v_{\bs,k+1} \in \bs^{[m]} = \bt^{[m]}$.  Consequently $v_{\bs,k+1} \in   \bt^{[m]} \setminus \{ v_{\bt,0}, v_{\bt,1},\dots,v_{\bt,k} \}$, from which it follows that $v_{\bt,k+1} \leq v_{\bs,k+1}$.  Therefore $ v_{\bs,k+1} = v_{\bt,k+1}$, from which it follows recursively that 
\[ (v_{\bt,0}, v_{\bt,1},\dots,v_{\bt,j}) = (v_{\bs,0}, v_{\bs,1},\dots, v_{\bs,j}),\]
as desired.  The claim for the $w$'s follows by a similar argument.
\end{proof}

\begin{corollary}\label{cor fringe}
Suppose that $\bt_n\ \in \fT_f$, $\bt \in \fT_\infty$ and $\bt_n\to \bt$.  If $v\in \bt$ is not on the infinite spine of $\bt$ then for sufficiently large $n$, $v\in \bt_n$ and $(\bt_n)_v = \bt_v$, where recall that $(\bt_n)_v$ is the fringe subtree of $\bt_n$ at $v$.
\end{corollary}

\begin{proof}
Suppose that $v$ is to the left of the infinite spine of $\bt$ so that $v=v_{\bt,k}$ for some $k$.  Let $k^*$ be such that $v_{\bt,k^*}$ is the smallest vertex on the infinite spine of $\bt$ that is larger than $v_{\bt,k}$.  Since $v_{\bt,k}$ is not on the inifnite spine of $\bt$, $v_{\bt,k^*} \notin \bt_{v_{\bt,k}}$ and, since the sequence $(v_{\bt,j})_{j\geq 0}$ lists the vertices of $\bt$ on or to the left of the infinite spine of $\bt$ in increasing lexicographic order, $v_{\bt,k^*}$ is a lexicographic upper bound on the vertices of $\bt_{v_{\bt,k}}$ and, indeed, on the set of vertices of the Ulam-Harris tree $U_\infty$ of the form $v_{\bt,k}u$ for $u\in \cU$.  Letting
\[ m^*=1+ \max_{0\leq i\leq k^*} |v_{\bt,i}|\]
it follows from Lemma \ref{lem reduction} that if $n$ is large enough that $\bt_n^{[m^*]} = \bt^{[m^*]}$ then
\[(v_{\bt_n,0}, v_{\bt_n,1},\dots,v_{\bt_n,k^*}) = (v_{\bt,0}, v_{\bt,1},\dots, v_{\bt,k^*}),\]
Consequently, $v_{\bt_n,k}= v_{\bt,k}=v$ and, since $v_{\bt,k^*}$ is a lexicographic upper bound on the set of vertices of the Ulam-Harris tree $U_\infty$ of the form $v_{\bt,k}u$ for $u\in \cU$, we see that all of the vertices of $(\bt_n)_{v_{\bt_n,k}} =(\bt_n)_{v_{\bt,k}}$ are in $\{v_{\bt_n,0}, v_{\bt_n,1},\dots,v_{\bt_n,k^*}\}$.  Therefore
\[\begin{split} (\bt_n)_{v_{\bt_n,k}} & = \left\{ w \in \{v_{\bt_n,0}, v_{\bt_n,1},\dots,v_{\bt_n,k^*}\} : w=v_{\bt_n,k}u \textrm{ for some } u \in \cU\right\}\\
& = \left\{ w \in \{v_{\bt,0}, v_{\bt,1},\dots,v_{\bt,k^*}\} : w=v_{\bt,k}u \textrm{ for some } u\in \cU\right\}\\
& = \bt_{v_{\bt,k}}.\end{split}\]
If $v$ is to the right of the infinite spine, we apply the above argument to the tree $\hat \bt$ obtained by reversing the left-to-right order of the children of each vertex in $\bt$.
\end{proof}

We will also need the following result about vertices on the infinite spine of $\bt\in \fT_\infty$.

\begin{lemma} \label{lem spine}
Suppose that $\bt_n\ \in \fT_f$, $\bt \in \fT_\infty$ and $\bt_n\to \bt$.  If $v\in \bt$ is on the infinite spine of $\bt$ then for sufficiently large $n$, $v\in \bt_n$ and $|(\bt_n)_v| \to |\bt_v| =\infty$, where recall that $(\bt_n)_v$ is the fringe subtree of $\bt_n$ at $v$. 
\end{lemma}

\begin{proof}
Since $v\in \bt^{[|v|]}$, it follows that $v\in \bt_n$ for all sufficiently large $n$.  Fix $N\geq 1$.  Since $v$ is on the infinite spine of $\bt$, $\bt^{[|v|+N]}$ contains a lexicographically increasing path from $v$ to some vertex $u$ with height $|u| = |v|+N$ (specifically, the part of the infinite spine that goes from height $|v|$ to height $|v|+N$).  For all $n$ sufficiently large $\bt_n^{[|u|]} = \bt^{[|u|]}$ and, consequently, $\bt_n^{[|u|]}$ contains a lexicographically increasing path from $v$ to $u$.  This path contains $N+1$ vertices and is contained in $(\bt_n)_v$.  Thus $|(\bt_n)_v| \geq N+1$ for all $n$ sufficiently large, so that $|(\bt_n)_v| \to |\bt_v| =\infty$ as desired. 
\end{proof}

\begin{theorem}\label{thm deterministic}
Fix $\sigma \in S_3$. Suppose that $\bt_n\ \in \fT_f$, $\bt \in \fT_\infty$ and $\bt_n\to \bt$.  Then for each $k\in \N$, $\Phi^{\sigma}_{\bt_n}(k)\to \Phi^{\sigma}_{\bt}(k)$ in $\N^*$.
\end{theorem} 

\begin{proof}

Fix $k\in \N$.  Note that it follows readily from the definition of the local topology and that $\bt_n\to \bt$ if and only if for each $m\geq 1$, $\bt_{n}^{[m]} = \bt^{[m]}$ for all sufficiently large $n$.  

$ $

\noindent We start with the $\sigma={\bf 321}$ case, which is the most difficult.  Let $(l_1, l_2,\dots)$ be a list of the leaves to the left of the infinite spine of $\bt$ in increasing lexicographic order and let $(\eta_0,\eta_1,\dots)$ be the list of vertices on the infinite spine in increasing lexicographic order (where $\eta_0$ is the root of $\bt$).  For each $i$ let $b(i)$ be such that $\eta_{b(i)}$ is the last vertex on the infinite spine of $\bt$ that is on the path from the root to $l_i$.  Let 
\[ m_i = 1+ \Ht(\bt \setminus \bt_{\eta_{b(i)+1}}),\]
which is finite because $\eta_{b(i)+1}$ is on the infinite spine. 
There are two cases we must consider: the case when there exists $k^*$ such that $s_{k^*}-p_{k^*}+1=k$ and the case when there is no such $k^*$.

Suppose that there exists $k^*$ such that $s_{\bt,k^*}-p_{\bt,k^*}+1=k$.  Observe that, by Lemma \ref{lem reduction}, if $\bt_{n}^{[m_i]} = \bt^{[m_i]}$ then $s_{\bt_n,i} = s_{\bt,i}$ and $p_{\bt_n,i} = p_{\bt,i}$.  Consequently, if $n$ is so large that $\bt_{n}^{[m_{k^*}]} = \bt^{[m_{k^*}]}$ then
\[\Phi^{\bf 321}_\bt(k) = \Phi^{\bf 321}_{\bt}(s_{\bt,k^*}-p_{\bt,k^*}+1) = s_{\bt,k^*}=s_{\bt_n,k^*}=\Phi^{\bf 321}_{\bt_n}(s_{\bt_n,k^*}-p_{\bt_n,k^*}+1)=\Phi^{\bf 321}_{\bt_n}(k).  \]

Now suppose that there does not exist such a $k^*$.  Then $k\in \{s_{\bt,i}-p_{\bt,i}+1 : i\geq 1\}^c$.   For $\br \in \fT_f$ (respectively, $\br \in \fT_\infty$) and $1\leq j \leq L(\br)$ (respectively $1\leq j <\infty$) let
\[ \alpha_{\br,j} = |\{s_{\br,i}-p_{\br,i}+1 : 1\leq i\leq j\}^c\cap\{1,\dots,j\}| \]
and
\[ \beta_{\br,j} = \inf\{ \ell: |\{s_{\br,i} : 1\leq i\leq \ell\}^c\cap\{1,\dots,\ell\}| \geq \alpha_{\br,j}\}.\]
It follows from the fact that $i\mapsto s_{\br,i}-p_{\br,i}+1$ and $i\mapsto s_{\br,i}$ are increasing that if $\br\in \fT_f$ then
\[ \{s_{\br,i}-p_{\br,i}+1 : 1\leq i\leq j\}^c\cap\{1,\dots,j\} = \{s_{\br,i}-p_{\br,i}+1 : 1\leq i\leq L(\br)\}^c\cap\{1,\dots,j\}\]
and
\begin{multline*} \inf\{ \ell: |\{s_{\br,i} : 1\leq i\leq \ell\}^c\cap\{1,\dots,\ell\}| \geq \alpha_{\br,j}\} \\ = \inf\{ \ell: |\{s_{\br,i} : 1\leq i\leq L(\br)\}^c\cap\{1,\dots,\ell\}| \geq \alpha_{\br,j}\}.\end{multline*}
The analogous statements hold when $\br\in\fT_\infty$.  

Consequently, for $\br \in \fT_f\cup\fT_\infty$, the unique increasing bijection from $\{s_{\br,i}-p_{\br,i}+1 : i\geq 1\}^c$ to $\{s_{\br,i} : i\geq 1\}^c$ is  $j\mapsto \beta_{\br,j}$. Therefore if $i\in\{s_{\br,i}-p_{\br,i}+1 : i\geq 1\}^c$ then $\Phi^{\bf 321}_{\br}(i)=\beta_{\br,i}$ and, in particular, $\Phi^{\bf 321}_{\bt}(k) = \beta_{\bt,k}$.  Let $\zeta = \max(k,\beta_{\bt,k})$.  From Lemma \ref{lem reduction}, we see that if $\bt_n^{[m_\zeta]}=\bt^{[m_\zeta]}$ then for $j\leq \zeta$,
\[\{s_{\bt_n,i}-p_{\bt_n,i}+1 : 1\leq i\leq j\}^c=\{s_{\bt,i}-p_{\bt,i}+1 : 1\leq i\leq j\}^c\]
and
\[ \{s_{\bt_n,i} : 1\leq i\leq j\}^c= \{s_{\bt,i} : 1\leq i\leq j \}^c.\]
Consequently, $k\in \{s_{\bt_n,i}-p_{\bt_n,i}+1 : 1\leq i\leq k\}^c$, $\alpha_{\bt_n,k} = \alpha_{\bt,k} $, and $\beta_{\bt_n,k} =\beta_{\bt,k} $.  Therefore, for sufficiently large $n$, 
\[ \Phi^{\bf 321}_{\bt_n}(k)= \beta_{\bt_n,k} =\beta_{\bt,k} = \Phi^{\bf 321}_{\bt}(k).\]
This completes the proof that for each $k$, $\Phi^{\bf 321}_{\bt_n}(k)=\Phi^{\bf 321}_{\bt}(k)$ for all $n$ sufficiently large.

$ $

\noindent The case when $\sigma={\bf 123}$ follows immediately from the case when $\sigma={\bf 321}$ since $\Phi^{\bf 321}_{\bt}(k)<\infty$ for all $k$.

$ $

\noindent Now consider $\sigma={\bf 231}$.  If follows from Lemma \ref{lem reduction} that $v_{\bt_n,k}=v_{\bt,k}$ for sufficiently large $n$.  From Corollary \ref{cor fringe} or Lemma \ref{lem spine}, depending on whether or not $v_k$ is on the infinite spine of $\bt$, that 
\[ |(\bt_n)_{v_{\bt_n,k}}| = |(\bt_n)_{v_{\bt,k}}| \to |(\bt)_{v_{\bt,k}}|.\]
Thus $\Phi^{\bf 231}_{\bt_n}(k)\to\Phi^{\bf 231}_{\bt}(k)$ as desired.

$ $

\noindent Now consider $\sigma={\bf 213}$. If $v_{\bt,k}$ is on the infinite spine of $\bt$, then $\bt\setminus \bt_{v_{\bt,k}}$ is a finite tree.  Let $m= 1+ht(\bt\setminus \bt_{v_{\bt,k}})$.  If $n$ is sufficiently large that $v_{\bt_n,k} = v_{\bt,k}$ and $\bt_n^{[m]} = \bt_n^{[m]}$, then arguing as in Corollary \ref{cor fringe} shows that
\[ \bt\setminus \bt_{v_{\bt,k}} = \bt_n\setminus \bt_{v_{\bt_n,k}}.\]
If $v_{\bt,k}$ is not on the infinite spine of $\bt$, then $\bt\setminus \bt_{v_{\bt,k}}$ is an infinite tree.  If $n$ is sufficiently large that $v_{\bt_n,k} = v_{\bt,k}$ then arguing as in Lemma \ref{lem spine} shows that
\[ |\bt_n\setminus \bt_{v_{\bt_n,k}}| \to |\bt\setminus \bt_{v_{\bt,k}}| =\infty.\]
In both cases, we see that $\Phi^{\bf 213}_{\bt_n}(k)\to \Phi^{\bf 213}_{\bt}(k)$, as desired.

$ $

\noindent The case when $\sigma={\bf 312}$. is similar to the case when $\sigma={\bf 231}$.

$ $

\noindent The case when $\sigma={\bf 132}$. is similar to the case when $\sigma={\bf 231}$.
\end{proof}

\section{Local limits of pattern avoiding permutations}
Let $T_{n+1}$ be a uniformly random element of $\fT_{n+1}$ and let $T$ be a $Geometric(1/2)$-Galton-Watson tree.  It is well known, and easy to verify, that $T_{n+1}$ is distributed like $T$ conditioned to have exactly $n+1$ vertices.  Let $\tilde T$ be a size-biased $Geometric(1/2)$-Galton-Watson tree.  

\begin{theorem}\label{thm main precise}
Fix $\sigma \in S_3$ and for each $n$ let $\Pi_n$ be a uniformly random element of $\avn{\sigma}$, considered as a random element in $F(\N,\N^*)$.  Then
\[ \Pi_n \overset{d}{\longrightarrow} \Phi^\sigma_{\tilde T}.\]
\end{theorem}

\begin{proof}
From Theorem \ref{thm bij1} or Theorem \ref{thm bij2} depending on $\sigma$, we see that $\Pi_n \overset{d}{=}  \Phi^\sigma_{T_{n+1}}$.  From Lemma \ref{lemma local limit}, we see that $T_{n+1} \overset{d}{\longrightarrow} \tilde T$ and, from the Skorokhod representation theorem, we may assume this convergence happens almost surely.  From Lemma \ref{lemma sizebp} we see that $\P(\tilde T \in \fT_\infty)=1$.  The result now follows from Theorem \ref{thm deterministic}.
\end{proof}

\bibliographystyle{plain}
\bibliography{pattern}

\end{document}